\documentclass[12pt]{amsart}
\usepackage[final]{graphicx}
\usepackage{amsfonts}
\usepackage{pdfsync}
\usepackage{amsmath}
\usepackage{amssymb}
\usepackage{amsthm}

\newtheorem{theorem}{Theorem}[section]
\newtheorem{lemma}[theorem]{Lemma}

\newtheorem{remark}[theorem]{Remark}
\newtheorem*{remarks}{Remarks}

\newcommand{\Sp}{{\mathord{\mathbb S}}}

\usepackage{color}

\def\ii{\'\i}

\begin{document}
\title[]{\textbf{Monotonicity of the period of a non linear oscillator}}

\author[Benguria]{Rafael~D.~Benguria$^1$}

\author[Depassier]{M.~Cristina~Depassier$^2$}

\author[Loss]{Michael~Loss$^3$}

\address{$^1$ Instituto  de F\'\i sica, Pontificia Universidad Cat\' olica de Chile,}
\email{{rbenguri@fis.puc.cl}}

\address{$^2$ Instituto  de F\'\i sica, Pontificia Universidad Cat\' olica de Chile,}
\email{{mcdepass@fis.puc.cl}}


\address{$^3$ School of Mathematics, Georgia Institute of Technology}
\email{{loss@math.gatech.edu}}

\begin{abstract} We revisit the problem of monotonicity of the period function
for the differential equation $u''-u+u^p=0$ and give a simple proof of recent
results of Miyamoto and Yagasaki. 

\end{abstract}

\maketitle

\section{Introduction} \label{intro}

The differential equation
\begin{equation}\label{eq:prob}
-\frac{d^2 u}{d x^2} + \frac{\lambda}{p-1} (v - v^p) = 0 \ ,
\end{equation}
despite its simple form, plays an interesting role in the theory of ordinary
differential equations and in the calculus of variations. When considered on the whole line
it is  the Euler - Lagrange equation for the sharp constant for a one dimensional Gagliardo Nirenberg inequality and can be solved explicitly. (For this and related results see \cite{MR3263963}). When considered on the circle, i.e.,
an interval with periodic boundary conditions, the equation exhibits interesting bifurcation
behavior. More precisely, fix $p >1$ and consider the problem of minimizing 
\begin{equation}\label{eq:functional}
\frac{\int_{\Sp^1} |u'(x)|^2 dx + \frac{\lambda}{p-1} \int_{\Sp^1} |u(x)|^2 dx}{\left(\int_{\Sp^1} |u(x)|^{p+1} dx \right)^{2/(p+1)}} \ .
\end{equation}
It is easily seen that a minimizer exists and that it satisfies an equation of type \eqref{eq:prob}.
It can be shown that for $ \lambda \le 1$,  $u = {\rm constant}$ is the only solution,
while for $\lambda >1$ there is an additional non-constant solution (see, e.g., \cite{MR3229793}). Let us mention that the same problem on the $d$-dimensional sphere was treated in \cite{MR1134481, MR1213110}, where it is shown that the constant solution is the only solution for $\lambda \le d$. For $\lambda > d$ the minimizing solution is not constant  and not much is known about it.

One way to understand the result in one dimension is through the period of the solutions.
If $\lambda$ is small one expects that the period of any non constant solution is large compared to the circumference of the circle and hence the circle can only support the constant
solution. If $\lambda$ increases one expects that the period decreases and hence at some
point a second solution bifurcates that is periodic and non-constant but fully supported by the circle. This new solution is then the minimizer of the functional.

Such problems can be very effectively studied using the period function. This idea,
to our knowledge, goes back to a paper by Smoller and Wasserman \cite{MR607786}
who use the period function to derive the complete bifurcation diagram
for such type of equations with cubic non-linearities. Schaaf
\cite{MR814016} proved the monotonicity of the period as a function of the energy for a class of Hamiltonian systems. This work was extended by Rothe \cite{MR1199531}. 
Chow and Wang \cite{MR833153} developed alternative formulas
that allowed them to prove the monotonicity of the period function for equations of the
type
$$
u''+e^u-1=0 \ .
$$
Some of the results of Chow and Wang were also obtained by Chicone \cite{MR903390}
using a different approach. For further references on the uses of the period function
the reader may consult \cite{MR3072678}.

It is somewhat surprising that, despite it ubiquity, the monotonicity of the period function for problem \eqref{eq:prob} in full generality was only established recently.  Chicone's work was the starting point for a thorough investigation of \eqref{eq:prob} by Miyamoto and Yagasaki \cite{MR2990036}
who proved  the monotonicity of the period function of \eqref{eq:prob} for integer $p$.
This result was then later generalized by Yagasaki \cite{MR3072678} to all values of $p >2$.
The approach in both of these works is to verify the Chicone Criterion which,
while non-trivial, is not too difficult for the case when $p$ is an integer.
The problem is, however, surprisingly difficult when $p$
is not an integer.  Yagasaki first treats the case where $p$ is rational and then extends 
the result to the general case by continuity.  Yagasaki's treatment is a real tour de force
and involves substantial amount of ingenuity and computations.
A consequence of Yagasaki's result and also established in \cite{MR3072678} is a complete
bifurcation diagram for the system \eqref{eq:prob} with Neumann Boundary conditions $v'(\pm \pi) = 0$.

In view of the complexity of the arguments in \cite{MR3072678} it is our aim to revisit this
problem and give proofs that are, in our view, more elementary than the ones given in
\cite{MR3072678}. Once more we start with Chicone's criterion which amounts to
check the convexity of a particular function. With simple changes of variables, this problem
is then recast in terms of solutions of differential equations that can 
be understood via maximum principles.

\section{The period function and Chicone's Criterion}
By rescaling the solution of \eqref{eq:prob}
$$
u(x) \to u\left( \sqrt{\frac{\lambda}{p-1}} x\right)
$$
we may assume that $u$ is a solution of the equation
\begin{equation} \label{eq:prob2}
u'' +u-u^p=0
\end{equation}
with periodic boundary conditions on the interval $[0, T]$
where
\begin{equation} \label{eq:interval}
T = 2\pi  \sqrt{\frac{\lambda}{p-1}} \ .
\end{equation}
Integrating this equation we get
\begin{equation}\label{eq:energy}
\frac{u'^2}{2}   +  V(u) = E
\end{equation}
where
\begin{equation} \label{eq:potential}
V(u) = \frac{u^{p+1}}{p+1}  - \frac{u^2}{2}  -\left( \frac{1}{p+1} -\frac{1}{2} \right)
\end{equation}
is the potential. Note that $V(u)$ has a minimum at $u=1$ and vanishes at that point.
The period function is given by
\begin{equation} \label{period}
T(E) = 2 \int_{u_-}^{u_+} \frac{1}{\sqrt{E-V(u)}} du
\end{equation}
where the values $u_{\pm}$ are determined by $V(u_\pm) = E$.

As in \cite{MR3072678} we start with Chicone's Criterion, which states that if for $u>0$ and $u \not= 1$
\begin{equation} \label{eq:chiccrit}
\left( \frac{V(u)}{V'(u)^2} \right)'' \ge 0
\end{equation}
then the function $T(E)$
is monotone increasing for $E$ in the interval $(0, ( \frac{1}{2} -\frac{1}{p+1}    ))$. Here, $u_\pm $ are the values
so that $V(u_\pm) = E$. The whole problem, originally solved in \cite{MR3072678}, is to verify
\eqref{eq:chiccrit} for the potential \eqref{eq:potential}.

\begin{remark} Near $u=1$ the potential is of the form
$$
V(u) = (p-1) (u-1)^2 +o((u-1)^2)
$$ 
and a straightforward computation shows that
$$
\lim_{E \to 0} T(E) =  \frac{2\pi}{\sqrt{p-1}} \ ,
$$
and strict monotonicity of $T(E)$ implies that
$$
T(E) >  \frac{2\pi}{\sqrt{p-1}} \ .
$$
Comparison with \eqref{eq:interval} implies  that a non-trivial periodic solution can only exist if $\lambda > 1$. 
\end{remark}

\noindent
\section{Proof of the monotonicity of the period function}
\begin{theorem}
The period function  \eqref{period} for the equation \eqref{eq:prob2} is monotonically increasing as a function of the energy $E$  for $0 < E < (p-1)/(2(p+1))$.
\end{theorem}
The Chicone criterion establishes that the period is an increasing function of the energy if 
\begin{equation}
\left( \frac{V(x)}{V'(x)^2} \right)' 
\end{equation}
is an increasing function of $x$ and provided that the potential is such that $x=0$ is the minimum of the potential and $V(0) =0$. Thus, we  shift the potential, i.e.,  we set $ W(\tilde u) =  V(\tilde u +1)$, and consider 
$$
W(u) = -\frac{1}{2} (u+1)^2 + \frac{ (1+ u)^{p+1}}{p+1} + \frac{p -1}{ 2 (p+1)},
$$
where we have omitted the use of tilde in $u$. We must prove then that
\begin{equation}
C(u) = \left( \frac{W(u)}{W'(u)^2} \right)'  \quad \text{increases with $u$.}
\end{equation}
It is convenient to make the change of variables
$$
u = e^t-1
$$
Since $du/dt>0$, $C(u)$ increasing with $u$ is equivalent to $C(t)$ increasing in $t$. With this change of variables $C(t)$ can be written as
\begin{equation}
C(t) =  - \frac{(p-1)}{(p+1)}   f_p(t)  e^{ - (p+3) t/2}
\end{equation}

where
\begin{equation} \label{eq:fp}
f_p(t) = \frac{ \sinh( p t) - p \sinh(t)}{4 \sinh^3 \left(  \frac{p-1}{2} t \right) } \ .
\end{equation}
We will show that $C(t)$ is a monotone increasing function of $t$.

\bigskip
\noindent
First we prove three elementary lemmas about the function $f_p(x)$.
\begin{lemma} \label{mainlemma}
For any $1<p<\infty$, the function $f_p(t)$ is even and positive, and one has

\bigskip

\noindent
i) $f_3(t) = 1$, all $t \in \mathbb{R}$, 

\bigskip
\noindent
ii) $f_p'(t) < 0$, all $0<t< \infty$, for $p>3$, and

\bigskip
\noindent
iii) $f_p'(t) > 0$, all $0<t< \infty$, for $1< p<3$.

\end{lemma}

\begin{proof}
The fact that $f_p(t)$ is even and positive for $p>1$ is obvious (the fact that $\sinh(p \,t) - p \sinh(t)$ for all $p>0$ and $t>0$ follows immediately from the monotonicity 
of the function $\sinh{t}/{t}$ for $y>0$. Using that $\sinh(t) = (e^t - e^{-t})/2$, i)  follows at once. 
To prove ii) and iii), notice first that 
\begin{equation}
f_p(t) = \frac{1}{3} \frac{ p(p+1)}{(p-1)^2} \left[ 1 - \frac{t^2}{40} (3 p -1) (p-3) + O(t^4) \right] \qquad \textrm{near} \qquad t=0,
\label{eq:f1}
\end{equation}
and 
\begin{equation}
f_p(t) \approx {\rm e}^{ -\frac{1}{2} (p-3) t} , 
\label{eq:f2}
\end{equation}
near infinity. We first prove ii), i.e., the case $p>3$: 

\bigskip
\noindent
Rewriting (\ref{eq:fp}) as
\begin{equation}
4 \sinh^3 \left(  \frac{p-1}{2}\,  t \right)  f_p(t) =  \sinh( p \, t) - p \sinh(t) 
\label{eq:f3}
\end{equation}
taking its derivative and multiplying by $\sinh((p-1)\, t/2)$ we obtain,  after using (\ref{eq:fp}) to eliminate $f_p$, 
\begin{eqnarray}
4 f'_p(t) \sinh^4 \left(  \frac{p-1}{2} t \right) + & \frac{3}{2} (p-1) \cosh \left(  \frac{p-1}{2} t \right) (\sinh( p \,t) - p \, \sinh t)  \nonumber \\
&= p (\cosh( p \, t) - \cosh t) \sinh \left(  \frac{p-1}{2} t \right).
\label{eq:f4}
\end{eqnarray}
Therefore, to prove $f'_p <0$ for $t>0$ and $p>3$  we need to show that
\begin{equation}
\frac{3}{2}(p-1) \cosh \left( \frac{p-1}{2} t \right) ( \sinh p\, t - p \,  \sinh t) >  p \sinh \left( \frac{p-1}{2} t \right) (\cosh( p \, t) - \cosh t).
\label{eq:f5}
\end{equation}
provided $p>3$ and $t>0$. 
Using the identity
$$
\cosh (\alpha t) \sinh (\beta t) = \frac{1}{2}  \sinh[(\alpha + \beta) t] +  \frac{1}{2}  \sinh[( \beta - \alpha ) t] 
$$
equation (\ref{eq:f5}) is equivalent to proving that 
\begin{equation}
h(t) \equiv  \frac{(p-3)}{4} \sinh \left( \Omega_3 \, t \right) +  \frac{3 p^2-p}{4} \sinh \left( \Omega_1 \,  t \right) -  \frac{3 p^2-10 p +3}{4}\sinh \left(\Omega_2 \, t \right)  >0,
\label{eq:f6}
\end{equation}
where we have conveniently defined the three ``frequencies''
\begin{equation}
\Omega_1(p) = \frac{p-3}{2},\quad  \Omega_2(p) = \frac{p+1}{2}, \quad \Omega_3(p) = \frac{3p-1}{2}.
\label{eq:freq}
\end{equation}
By direct calculation, one can show that $h(t)$ satisfies the second order differential equation
\begin{equation}
h''(t) - \Omega_3^2 \,  h(t) = p (p-1)(3p-1) \left[  (p-3)  \sinh \left( \frac{p +1}{2} t \right) - (p+1) \sinh \left( \frac{ p -3}{2} t \right) \right].
\label{eq:f7}
\end{equation}
Notice that for $p>3$ and $t>0$ the right side of (\ref{eq:f7}) is positive.
In fact, defining  $r= p+1, s= p-3$, the right side is positive since
$$
\frac{\sinh ( r \, t)}{r} > \frac{\sinh ( s \, t)}{s} \qquad \textrm{if}\qquad r>s,
$$
which follows immediately from the montonicity of $\sinh{y}/y$ for $y>0$. 

\bigskip
\noindent
To prove that $h(t)$ is  a positive  function (even more, it is an increasing function) of $t \in (0, \infty)$, we first notice that
\begin{equation}
h(t) = \frac{1}{240}  (p-3)(3p-1)(p-1)^2 p (p+1)t^5 + O(t^6),
\label{eq:f8}
\end{equation}
that is, for $p>3$, and $t>0$, $h(t)$ is positive in a neighborhood of zero (for $t \neq 0$). 
Since $h$ is a continuous function of $t$, if $h$ were to become negative it should reach a local maximum at some point $t_1>0$, say. But at $t_1$, one has at the same time 
$h''(t_1) < 0$, and $h(t_1) >0$. Thus, at $t_1$, the left side of (\ref{eq:f7}) is strictly negative whereas the right side is strictly postive, which is a contradiction. This in fact shows that $h(t)$ is not only positive but also it is strictly increasing in $t$. Usng (\ref{eq:f4}),  (\ref{eq:f5}), and  (\ref{eq:f6}), this in turn  implies ii). 

\bigskip
\noindent
Finally, in order to prove iii) we use the same argument. This time, however, since $1<p<3$ it follows that the right side of (\ref{eq:f4}) is strictly negative for $t \in (0,\infty)$. 
Also this time, it follows from (\ref{eq:f8}) that $h(t)$  is negative  in a neighborhood of zero (for $t \neq 0$). Now we have  from (\ref{eq:f7}) that $h(t)$ is negative (in fact 
strictly decreasing) for $t>0$, for if this were not the case we would have a point $t_2$ which would be a negative local minimum of $h$. At that point  we would have 
$h''(t_2) > 0$ and $h(t_2) <0$. 
Hence, at $t_2$, the left side of (\ref{eq:f7}) would be strictly positive and at the same time the right side would be strictly negative, which is again a contradiction. 
This concludes the proof of iii) and the proof of the lemma. 

\end{proof}

\begin{lemma} \label{positive}
For any $1<p<\infty$, 
\bigskip

\begin{equation}
f_p'(t) + \frac{(p+3)}{2} f_p(t)  >0,
\label{eq:lemma2res}
\end{equation}
all $t >0 $.
\end{lemma}
\begin{remarks} 

\noindent
i) In fact one can prove that, 
for $t>0$, $f_p'(t) + \frac{(p+3)}{2} f_p(t)$ is constant (equal to $(p+3)/2$)  for $p=3$, is strictly decreasing for $p>3$, and is strictly increasing for $1<p<3$. Since 
we do not need these facts in the sequel, we omit their proof. 

\noindent
ii) In view of the previous lemma, (\ref{eq:lemma2res}) is certainly true for $1<p \le 3$. 

\end{remarks}

\begin{proof} Proceeding as in the proof of the previous lemma, it is convenient to introduce 
\begin{equation}
H(t) = \sinh^4 \left(\frac{(p-1)t}{2}\right) \left[f_p'(t)  + \frac{(p+3)}{2} f_p(t) \right].
\label{eq:f9}
\end{equation}
We need to prove that $H(t) >0$ for $t\in (0, \infty)$.  
From (\ref{eq:f9}) and the definition of $f_p(t)$ one can verify that
\begin{equation}
H(t) =  \frac{1}{96} p (p-1)^2 (p+1) (p+3) t^4  + O(t^5)
\label{eq:f10}
\end{equation}
and therefore $H(t)>0$ in a neighborhood of zero (for $t \neq 0$).  
Moreover, 
\begin{equation}
H(t) \approx  \frac{3}{16} \exp(\Omega_3 \, t) >0, 
\label{eq:f11}
\end{equation}
as $t \to \infty$. 
It is straightforward to show that $H(t)$ satisfies, 
\begin{equation}
H''(t) - \Omega_3^2 \, H(t) = p (p-1) S(t),
\label{eq:f12}
\end{equation}
with
\begin{align} \label{eq:f13} 
S(t) \equiv &\frac{1}{8}[(p+3)(p+1) \left(  \cosh(\Omega_2 \, t) -  \cosh(\Omega_1 \,  t)\right) 
 \\ 
 &-(3p-1)(p-3)\sinh(\Omega_2 \,  t) + (3 p-1) (p+1) \sinh(\Omega_1 \,  t)] \ . \nonumber
\end{align}
Expanding $S(t)$ around $t=0$ we have that
\begin{equation}
S(t) = \frac{ (p+1)(p+3)(p-1)}{8} t^2 + O(t^3),
\label{eq:f14}
\end{equation}
in a neighborhood of $0$. Moreover, 
\begin{equation}
S(t) \approx  -\frac{p(p-7)}{8} e^{\Omega_2  \, t},
\label{eq:f15}
\end{equation}
near infinity, for $p>1$ and $p \neq 7$. A simple computation shows that $S(t) \approx (5/12) e^{2t}$ near infinity when $p=7$. 
Now, it follows from (\ref{eq:f13}) that $S(t)$ satisfies the following ordinary differential equation, 
\begin{equation}
S''(t) - \Omega_1^2  \, S(t) = \frac{1}{4} \cosh\left(\Omega_2 \, t \right) G(t), 
\label{eq:f16}
\end{equation}
with
\begin{equation}
G(t) = (p+3) (p+1) - (3p-1) (p-3) \tanh (\Omega_2  \, t )
\label{eq:f17}
\end{equation}

\bigskip
\noindent
Since $\tanh(y)$ is an increasing function of $y$, $G(0) =(p+3)(p+1)$, and $G(\infty) = - 2 p(p-7)$, 
it follows at once from (\ref{eq:f17}) that 

\bigskip
\noindent
i) $G(t)$ is a positive and increasing function of $t$ for $t \in (0,\infty)$, when $1<p<3$. 

\bigskip
\noindent
ii) $G(t)=G(0)= 24$, for all $t$, for $p=3$, 

\bigskip
\noindent
iii) $G(t)$ is a positive and decreasing function of $t$ for $t \in (0,\infty)$, when $3<p\le 7$, and 

\bigskip
\noindent
iv) $G(t)$ is a decreasing function of $t$ for $t \in (0,\infty)$, and has only one zero $t_3$ when $p>7$ given by  $\tanh(\Omega_2 \,  t_3) = (p+3)(p+1)/[(3p-1)(p-3)]$. 

\bigskip
\noindent
Using these properties of $G(t)$ and the equation (\ref{eq:f16}),  it is simple to prove the following properties of $S(t)$:

\bigskip
\noindent
S1)  For $1<p \le 7$, and $t \in (0,\infty)$, $S(t)$  is positive. In fact, it is strictly increasing. 

\bigskip
\noindent
S2) For $7<p$, and $t \in (0,\infty)$, $S(t)$  has only one zero, say $\hat t$, so that $S(t)>0$ in $(0,\hat t)$, whereas $S(t) < 0$ (and strictly decreasing) for $t>\hat t$.  

\bigskip
\noindent
It follows from (\ref{eq:f14}) that $S(t)$ is positive in a neighborhood of $0$. Since $S$ is a continuous function, if $S(t)$ would be negative at some point, there 
has to be a local maximum of $S(t)$, say $t_m$. At this point one would have $S''(t_m) <0$, and $S(t_m)>0$. But this yields a contradiction with (\ref{eq:f16}) since 
$G(t)>0$ for all $t \in (0, \infty)$ when $1<p \le 7$. This proves S1). 

\bigskip
\noindent
On the other hand, if $p>7$, it follows from (\ref{eq:f15}) that $S(x) >0$ for large values of $x$, and we recall from (\ref{eq:f14}) that $S(x)>0$ in a neighborhood of zero. 
If $S(x)$ were not positive for all $x \in (0,\infty)$ there should exist two points $0<x_a<x_b<\infty$, such that $x_a$ is a local positive maximum of $S$ and $x_b$ a negative local minimum
of $S$. Using (\ref{eq:f16}) this in turn implies $G(x_a) <0$, and $G(x_b)>0$. From here it finally follows that $x_a > x_3$ and $x_b <  x_3$, which is a contradiction with the fact that
$0<x_a<x_b<\infty$. This proves S2). 

Using the properties S1) and S2) of $S(x)$ we can now conclude the proof of the lemma, i.e., that $H(x)>0$. We proceed exactly as above. If $1<p \le 7$, $S(x)>0$ in $(0,\infty)$. From (\ref{eq:f10}) and (\ref{eq:f11}) respectively, $H(x)$ is positive near zero and at infinity. Since $H$ is continuous, if $H$ were to become negative somewhere in $(0,\infty)$ there would exist a positive local maximum of $H$, say $x_c$. This would imply that the left side of (\ref{eq:f12}) at $x_c$  would be negative. But we know that for all $x \in (0,\infty)$ the 
right side of (\ref{eq:f12}) is positive, which is a contradiction. On the other hand, for $p>7$, $S(x)$ has a unique zero, $\hat x$, in $(0,\infty)$. If in this case $H$ would be negative somewhere in $(0,\infty)$, there should be two points, $0 <x_d < x_e < \infty$ say, with $x_d$ a positive local maximum of $H$ and $x_e$ a negative local minimum of $H(x)$. 
From (\ref{eq:f12}) and the property S2) it follows that $x_d > \hat x$ and $x_e < \hat x$ which is a contradiction with the fact that  $0 <x_d < x_e < \infty$. This concludes the proof of the lemma. 
\end{proof}

\begin{lemma} \label{negative}
For any $1<p<\infty$, 
\bigskip

\begin{equation}
f_p'(t) -  \frac{(p+3)}{2} f_p(t)  < 0,
\label{eq:lemma3res}
\end{equation}
all $t >0$.
\end{lemma}
\begin{remarks} 
Because of Lemma \ref{mainlemma},  (\ref{eq:lemma3res}) is true for $3 \le p < \infty$, thus we need only prove it for $1 < p < 3$. 
\end{remarks}

\begin{proof} 
The proof is analogous to the proof of the previous lemma. We begin by introducing
\begin{equation} 
J(t) = \sinh^4\left( \frac{p-1}{2}t \right) \left( f_p'(t) - \frac{p+3}{2} f_p(t) \right).
\label{eq:g1}
\end{equation}
and aim to prove that $J(t)<0$, for all $0<t <\infty$ and for all $1<p<3$. 
Expanding $J(t)$ around $t=0$ we find, 
\begin{equation}
J(t) =  -\frac{1}{96} p(p+1)(p+3)(p-1)^2 t^4 + \mathcal{O} (t^5) 
\label{eq:g2}
\end{equation}
hence, $J(t)$ is negative in a neighborhood of $0$. 
It follows from (\ref{eq:g1}) that  $J(t)$ satisfies the differential equation
\begin{equation}
J''(t) - \Omega_3^2 J(t) = p(p-1) T(t),
\label{eq:g4}
\end{equation}
where
\begin{align}\label{eq:g5}
 T(t) = &\frac{(p+1)(p+3)}{8} \left(\cosh (\Omega_1 \, t)  - \cosh( \Omega_2 \, t)\right)   \\
& +\frac{(3 p-1)(p+1)}{8} \sinh (\Omega_1 \, t)-\frac{(p-3)(3 p-1)}{8} \sinh( \Omega_2 \, t)\ .
\nonumber
\end{align}
It follows immediately from (\ref{eq:g5}) that 
\begin{equation}
T''(t) - \Omega_1^2 T(t) = - \frac{1}{4} (p-1) \cosh \Omega_2 t \left[ (p+1) (3 + p) + (p-3) (3p-1) \tanh (\Omega_2 \, t)  \, \right].
\label{eq:g6}
\end{equation}
Since $1<p<3$, the function  $(p+1) (3 + p) + (p-3) (3p-1) \tanh (\Omega_2  \, t)$ is decreasing in $t$, so it is everyhere bigger than 
its value at infinity, i.e., than $4 p^2 +6 (1-p) > 4$. Hence, the right side of (\ref{eq:g6}) is negative for all $t>0$ and $1<p<3$. 
From (\ref{eq:g2}) and (\ref{eq:g4}) we have that
\begin{equation}
T(t) =  -\frac{1}{8} (p+3)(p^2-1)  t^2 + \mathcal{O} (t^3),
\label{eq:g7}
\end{equation}
i.e., $T(t)$ is negative in a neighborhood of $0$. Since the right side of (\ref{eq:g6}) is strictly negative, it then follows  that $T(t) <0$ 
(and, in fact decreasing) for all $t>0$ and $1<p<3$. For, if this were not the case there would be a local negative minimum, at $s$ say, where 
$T''(s)>0$ and $T(s)<0$. Thus, at $s$ the left side of (\ref{eq:g6})  would be positive, which is a contradiction since we know the right side of that equation is negative for all $t>0$. 
From here, using this time (\ref{eq:g4}) and 
(\ref{eq:g5}) and exactly the same argument we conclude that $J(t) <0$ (and, in fact strictly decreasing) for all $t>0$ and $1<p<3$. This concludes the proof of the lemma. 

\end{proof}
\begin{proof}[Proof of the Theorem]
First notice that  for the exactly solvable case $p=3$, $C(t) = -(1/2) e^{- 3 t}$ increases with $t$, so for $p=3$ the period increases with the energy. For $p \not= 3$ distinguish two cases, $t >0$ and $t<0$.  For $t>0$ the theorem follows at once from from Lemma \ref{negative} by differentiation. For $t<0$, the theorem follows from the fact that $f_p$ is even and Lemma \ref{positive}.

\section{Acknowledgments}
\thanks{The work of R.B. and M.C.D. has been supported by Fondecyt (Chile) Projects \# 112--0836 and \#114-1155 and by the Nucleo Milenio en  ``F\ii sica Matem\'atica'', RC--12--0002 (ICM, Chile).
M.L. would like to thank the Pontificia Universidad Cat\' olica de Chile for their hospitality.
The work of M.L. has been supported in part by Fondecyt (\#114-1155), by the US National Science Foundation Grant DMS-1301555 and the Humboldt Foundation.}

\end{proof}


\end{document}